\documentclass[leqno,12pt]{article}

\usepackage{epsfig}

\usepackage{amsmath}
\usepackage{amscd}
\usepackage{amsopn}
\usepackage{amsthm}
\usepackage{amsfonts,amssymb}
\usepackage{amsfonts,bbm}
\usepackage{latexsym}

\makeatletter
\@addtoreset{equation}{section}
\makeatother

\newtheorem{lemma}{LEMMA}[section]
\newtheorem{proposition}[lemma]{PROPOSITION}
\newtheorem{corollary}[lemma]{COROLLARY}
\newtheorem{theorem}[lemma]{THEOREM}
\newtheorem{remark}[lemma]{REMARK}

\newtheorem{examples}[lemma]{EXAMPLES}

\newenvironment{knownresult}[1][THEOREM]{\begin{trivlist}
\item[\hskip \labelsep {\bfseries #1.}] \itshape}    {\end{trivlist}}

\newcommand{\real}{\mathbbm{R}}
\newcommand{\nat}{\mathbbm{N}}

\renewcommand{\a}{\alpha}
\renewcommand{\b}{\beta}

\newcommand{\g}{\gamma}

\newcommand{\vp}{\varphi}
\newcommand{\ve}{\varepsilon}

\newcommand{\reald}{{\real^d}}

\newcommand{\on}{\quad\text{ on }}
\newcommand{\und}{\quad\mbox{ and }\quad}
\newcommand{\inv}{^{-1}}
\newcommand{\ov}{\overline}

\newcommand{\dist}{\mbox{\rm dist}}

\newcommand{\itemframe}%
{\setlength{\parskip}{10pt}\begin{enumerate} \setlength{\topsep}{10pt}%
\setlength{\itemsep}{15pt}\setlength{\parsep}{5pt}}

\title{Champagne subregions with unavoidable bubbles}
\author{WOLFHARD HANSEN and IVAN NETUKA
\thanks{Both authors gratefully acknowledge support
by CRC-701, Bielefeld.}}

\date{}
\begin{document}
\maketitle 

\begin{abstract} 
A champagne subregion of a connected open  set $U\ne\emptyset$ in $\mathbbm R^d$, $d\ge 2$, is obtained
omitting  pairwise disjoint closed balls $\ov B(x, r_x)$, $x\in X$, the bubbles, 
where $X$ is a locally finite set in $U$. 
The union $A$ of  these balls may be unavoidable, that is, Brownian motion,  starting in $U\setminus A$
and killed when leaving~$U$,  may hit~$A$ almost surely or, equivalently, 
$A$ may have harmonic measure one for $U\setminus A$. 
 
Recent publications by  Gardiner/Ghergu ($d\ge 3$) and by Pres ($d=2$) give rather sharp 
answers to the question how small such a set $A$ may be, when $U$~is the unit ball.

In this paper, using a~new criterion for unavoidable sets 
and a~straightforward approach, even stronger  results are obtained, results which hold as well
for an arbitrary open set $U$.
\end{abstract}

\section{Introduction and main theorem}

Throughout this paper let $U$ denote a non-empty connected  open  set in $\reald$, $d\ge 2$. 
Let us say that a subset  $A$ of $U$ which is relatively closed in $U$ is  \emph{unavoidable},
if Brownian motion, starting in $U\setminus A$ and killed when leaving $U$,  hits $A$ almost surely or, equivalently,
if $\mu_y^{U\setminus A}(A)=1$,  for every $y\in U\setminus A$, where $\mu_y^{U\setminus A}$ 
denotes the harmonic measure at $y$  with respect to~$U\setminus A$.
\footnote{Let us note that $\mu_y^{U\setminus A}$ may fail to be a probability measure,
if $U\setminus A$ is not bounded.}

For $x\in \reald$ and $r>0$, let $B(x,r)$ denote the open ball of center~$x$  and radius~$r$. 
Suppose that $X$ is a countable set in~$U$ having no accumulation point in $U$, and let $r_x\in (0,1)$, $x\in X$, such that  
the closed balls $\ov B(x,r_x)$,  the \emph{bubbles}, are pairwise disjoint, $\sup_{x\in X} r_x/\dist(x,\partial U)<1$ and,
if $U$ is unbounded, $r_x\to 0$ as $x\to \infty$.   
Then the union $A$ of all $\ov B(x,r_x)$ is relatively closed in $U$, and the open connected set $U\setminus A$ 
is  called a~\emph{champagne subregion of~$U$}. 
This generalizes the notions used in~\cite{akeroyd, gardiner-ghergu, odonovan, ortega-seip,pres}
in the case, where $U$ is the unit ball; see also \cite{carroll} for the case, where $U$ is $\reald$, $d\ge 3$.   

It will be convenient  to introduce the set $X_A$ for a champagne subregion $U\setminus A$: 
$X_A$ is the set of centers of all 
the  bubbles forming $A$ (and $r_x$, $x\in X_A$, is the radius of the bubble centered at $x$).

The main result of Akeroyd \cite{akeroyd} is, for a given $\delta>0$,  the existence of  a~champagne subregion of the unit disc
such that 
\begin{equation}\label{aker}
\mbox{$\sum\nolimits_{x\in X_A}\  r_x<\delta$ and yet  $A$ is unavoidable.}
\end{equation} 
Ortega-Cerd\`a and Seip \cite{ortega-seip} improved
the result of Akeroyd in characterizing a certain class of champagne subregions $U\setminus A$ 
of the unit disc, where $A$ is unavoidable and $\sum_{x\in X_A} r_x<\infty$,  and hence the statement of (\ref{aker}) 
can be obtained omitting finitely many of the discs $\ov B(x,r_x)$, $x\in X_A$.

Let us note that already in \cite{HN-sigma} the existence of a champagne subregion  of an arbitrary 
bounded connected open set   $V$ in $\real^2$ having property (\ref{aker})  was crucial for the construction of an example
answering Littlewood's one circle problem to the negative. In fact, \hbox{\cite[Proposition 3]{HN-sigma}}  is a bit stronger:
Even a Markov chain formed by jumps on annuli hits $A$ before it goes to $\partial V$. The statement about
harmonic measure (hitting by Brownian motion) is obtained by the first part of the proof of \cite[Proposition 3]{HN-sigma}
(cf.~also \cite{H-kouty}, where this is explicitly stated at the top of page~72). 
This part  uses only  ``one-bubble-estimates''  for the global Green function and the minimum principle.

Recently, Gardiner/Ghergu \cite[Corollary 3]{gardiner-ghergu} proved the following. 

\begin{knownresult}[THEOREM A]  Suppose that  $U$ is the unit ball in $\reald$, $d\ge 3$. Then, for all $\a>d-2$  and $\delta>0$,
there exists a~champagne subregion $U\setminus A$   such that $A$ is unavoidable 
and 
$$ 
\sum\nolimits_{x\in X_A}\  r_x^\a <\delta.
$$ 
\end{knownresult}

Moreover, Pres \cite[Corollary 1.3]{pres} showed  the following for the plane.

\begin{knownresult}[THEOREM B]
Suppose that $U$ is the unit disc in $\real^2$. 
 Then, for all $\a>1$ and $\delta>0$, there exists a~champagne subregion $U\setminus A$ 
 such that $A$ is unavoidable and 
$$
\sum\nolimits_{x\in X_A} \ \bigl (\log\frac 1{r_x}\bigr)^{-\a}   <\delta.
$$ 
\end{knownresult}

Due to capacity reasons both results are sharp in the sense that $\a$ cannot be replaced by $d-2$ in Theorem A      
and $\a$ cannot be replaced by $1$ in Theorem B.  The proofs are quite involved and, in addition, use the delicate results
\cite[Theorem 1]{essen} (cf.\ \cite[Corollary 7.4.4]{aikawa-essen}) on minimal thinness of subsets $A$ of~$U$ at  points 
$z\in\partial  U$
and \cite[Proposition 4.1.1]{aikawa-borichev} on quasi-additivity of capacity. 

Aiming at a proof for the results by Gardiner/Ghergu and Pres, by using only elementary estimates for Green functions 
and the minimum principle, we have obtained statements which are  even much  sharper than the previous results and can even
be extended to arbitrary connected open sets  $U$.  In addition,
we are able to treat the cases $d\ge 3$ and $d=2$ in exactly the same way. Here is our main result  
(where $\log^{(n)}$ is recursively defined by $\log^{(1)}:=\log$ and $\log^{(n+1)}:=\log \log^{(n)}$).   

\begin{theorem} \label{main23}
Let $U\ne \emptyset $ be a connected open  set  in $\reald$, $d\ge 2$. Then, for all $n\in\nat$ and $\delta>0$,
there is a~champagne subregion $U\setminus A$           
such that $A$ is unavoidable and     
\begin{eqnarray*} 
\sum\nolimits_{x\in X_A} \ \bigl(\log \frac 1{r_x}\bigr)\inv  \bigl(\log^{(n+1)} \frac 1{r_x} \bigr)\inv  &<&\delta, \qquad  \mbox{ if } d=2,\\
\sum\nolimits_{x\in X_A} \ r_x^{d-2} \bigl(\log^{(n)} \frac 1{r_x}\bigr)\inv&<&\delta, \qquad \mbox{ if } d\ge 3. 
\end{eqnarray*} 
   \end{theorem} 

For a simultaneous discussion of the cases $d=2$ and $d\ge 3$, we define functions 
\begin{equation*} 
 N(t) :=\begin{cases} \log \frac 1t  ,&\quad\mbox{ if } d=2,\\
                                t^{2-d} ,&\quad\mbox{ if }d\ge 3,
                   \end{cases}                   \und \vp(t):= N(t)\inv
\end{equation*} 
so that  $(x,y) \mapsto N(|x-y|)$   is the global Green function  and, for $d\ge 3$, $\vp(t)=t^{d-2}$ is the capacity of a ball
with radius $t$ (and, for $d=2$, $\vp(t)$  should only be considered for $t\in (0,1)$). Using the (capacity) function $\vp$
our Theorem \ref{main23} adopts the following form.

\begin{theorem} \label{main}
Let $U\ne \emptyset $ be a connected open  set  in $\reald$, $d\ge 2$. Then, for all $n\in\nat$  and $\delta>0$,
there is a~champagne subregion $U\setminus A$  
such that $A$  is unavoidable and 
\begin{equation}\label{sharp}
\sum\nolimits_{x\in X_A}  \vp(r_x) \bigl(\log^{(n)} \frac1{\vp(r_x)}\bigr)\inv  < \delta.
\end{equation} 
 \end{theorem} 
       
Accordingly,  the results by Gardiner/Ghergu and Pres (Theorems A and B) can be unified as follows.    

\begin{knownresult}[THEOREM C] For all $\ve >0$ and $\delta>0$, there is a~champagne subregion $U\setminus A$ 
of the unit ball  $U$ in $\reald$, $d\ge 2$, such that $A$ is unavoidable 
and 
$$
\sum\nolimits_{x\in X_A} \ \vp(r_x)^{1+\ve}     < \delta.
$$
\end{knownresult}

The key to our result is a general  criterion for unavoidable sets (Section 2).
 We  next introduce a suitable exhaustion 
of the unit ball by open balls  $U_n$,  choose finite subsets $X_n$ in $\partial U_n$, and radii $r_n$ for bubbles $\ov B(x,r_n)$, 
$x\in X_n$, $n\in\nat$ (Section 3). To illustrate the  power of our criterion, we first combine
it with the ``one-bubble-estimate'' used in the proof of \cite[Proposition 3]{HN-sigma} (Section 4). Our Proposition
\ref{bubble} is already fairly close to Theorem C.                 
We then prove a very general result which immediately implies Theorem \ref{main} for the unit ball (Theorem \ref{best})). 
Finally, using the ingredients of this proof, we obtain Theorem \ref{main} (and more) in full generality (Section~6).

\section{A general criterion for unavoidable sets} \label{one-bubble}

 Given an open set $W$ in $\reald$ and a  bounded Borel measurable function~$f$ on $\reald$,   
let $H_Wf$ denote the function which extends  the (generalized) Dirichlet solution  \hbox{$x\mapsto \int f\,d\mu_x^W$}, 
$x\in W$, to a function on $\reald$ taking the values $f(x)$ for $x\in\reald\setminus W$.
We shall use that the harmonic kernel $H_W$ has the following property: If $W'$ is an open set in $W$,
then $H_{W'}H_W=H_W$. 

Let $U\ne \emptyset$ be a connected open set  in $\reald$, $d\ge 2$, and let $A\subset U$ be  relatively closed. 
Then $A$ is unavoidable (in $U$) if 
\begin{equation*} 
H_{U\setminus A}1_A=1 \on\  U.
\end{equation*}

\begin{proposition}\label{nlogn}
Let $0\le \g_n<1$ and  $V_n$  bounded open  in $U$,   $n\in \nat$,              
such that $\ov V_n\subset V_{n+1}$, $V_n\uparrow U$,  and the following holds:  For all  $n\in\nat$     
and $z\in \partial V_n\setminus A$, there exists a closed set~$E$ in~$A\cap V_{n+1}$ such that 
\begin{equation}\label{HUE}
H_{V_{n+1}\setminus E} 1_{E}(z)\ge     \g_n.
\end{equation} 
Then, for all $n,m\in \nat$, $n<m$,
\begin{equation}\label{criterion}
  H_{U\setminus A} 1_A\ge 1-\prod\nolimits_{n\le j<m} (1-\g_j) \on\ \ov V_n. 
\end{equation} 
In particular,  $A$ is unavoidable if the series $\sum\g_n$ is divergent.
\end{proposition} 

\begin{proof} For $j\in \nat$, let  $W_{j+1}:=V_{j+1}\setminus A$. If $E$ is a  closed set  in $A\cap V_{j+1}$, 
then $H_{W_{j+1} }1_{\partial V_{j+1}}\le 1- H_{V_{j+1}\setminus E} 1_{E}$, by  the minimum principle. Hence, by (\ref{HUE}), 
\begin{equation*} 
H_{W_{j+1} }1_{\partial V_{j+1}} \le  1-\g_j \on \partial V_j.
\end{equation*} 
Now let $n,m\in\nat$, $n<m$. By induction,
\begin{equation*} 
    H_{W_m} 1_{\partial V_m}
    =H_{W_{n+1} }H_{W_{n+2}}\dots H_{W_m} 1_{\partial V_m}
           \le \prod_{n\le j<m} (1- \g_j) \on \partial  V_n.
 \end{equation*} 
By the minimum principle,  we conclude that 
\begin{equation*} 
H_{U\setminus A} 1_A\ge      H_{W_m} 1_{ A} \ge  1-   H_{W_m} 1_{\partial V_m}\ge  1- \prod_{n\le j<m} (1- \g_j) \on\  \ov V_n.
\end{equation*} 
\end{proof}

\section{Exhaustion of the unit ball, choice of bubbles} 

Let $k_0\in\nat$ and, for $k\ge k_0$, let  $m_k\in\nat$ and $\a_k, \b_k\in (0,\infty) $    such that $m_k\le m_{k+1}$, 
$\a_{k+1}\le \a_k$, $\b_k\le \b_{k+1}$,    
\begin{equation*} 
      \a_k\le 1/\max\{k,m_k\} , \quad 1\le \b_k\le k,  \quad \sum_{k\ge k_0} m_k\a_k=\infty, \quad \sum_{k\ge k_0} m_k \a_k/\b_k<1/2. 
\end{equation*} 
For every $k\ge k_0$, let $a_k:=\a_k/\b_k$ and 
\begin{equation*} 
 U_k:=B(0, R_k)  \quad  \mbox{ with  } \quad       R_k:= 1-\sum\nolimits_{j=k}^\infty m_j a_j>\frac 12,
\end{equation*} 
so that $R_{k+1}-R_k=m_k a_k$.

\begin{remark}\label{best-choice} {\rm
For the proof of Theorem \ref{main} we shall take      
\begin{equation*} 
     m_k:=p_n(k), \quad  \a_k:= (k m_k)\inv, \und   \b_k:=k, 
\end{equation*} 
where $p_n\colon \nat\to\nat$ is recursively defined by 
$p_0(k):=k$ and $p_n(k):=2^{p_{n-1}(k)}$. Let us note that here
$R_{k+1}-R_k=1/k^2$.
}
\end{remark} 

For every $k\ge k_0$, we fix  a finite subset $X_k$ of $\partial U_k$   such that the balls $B(x,a_k/3)$, $x\in X_k$,        
 cover~$\partial U_k$ and the balls $B(x,a_k/9)$, $x\in X_k$, are pairwise disjoint. Such a~set~$X_k$ exists (see \cite[Lemma 7.3]{Rudin}).  
A consideration of the  areas involved,  when intersecting the balls with $\partial U_k$,       
shows that \footnote{We shall write $f\approx g$, if there exists a constant $C=C(d)>0$ such that $C\inv\le f\le Cg$.}
\begin{equation}\label{area}
 \# X_k \approx  a_k^{1-d}.
\end{equation} 
Further, let   
\begin{equation}\label{def-r}
          r_k:=\begin{cases} \exp(-\a_k\inv),&\quad\mbox{ if } d=2,\\[1.5mm]
                                       a_k \a_k^{1/(d-2)},&\quad\mbox{ if }d\ge 3.
                   \end{cases}
\end{equation} 
In other words,  we define  $r_k$  in such a way that  
\begin{equation}\label{ar}
              \vp(r_k)=a_k^{d-2}\a_k\le m_k^{1-d}.
\end{equation} 
Hence, by (\ref{area}), 
\begin{equation}\label{log-estimate}
 \# X_k\,  \vp(r_k)\approx a_k\inv \a_k=\b_k .
\end{equation} 
 
Finally, we take        
\begin{equation*}  
        r_x:=r_k,\qquad \mbox{ if } x\in X_k, \, k\ge k_0.
\end{equation*} 
Looking at  (\ref{def-r}) we  see that     
$   \lim\nolimits_{k\to \infty}  r_k/a_k = 0$ (to deal with the case $d=2$ we  note that $a_k=\a_k/\b_k\ge \a_k/k \ge \a_k^2$).
We fix  $k_1\ge k_0$ such that  
\begin{equation}\label{rnan}
 r_k<  a_k/100, \quad\mbox{  for every }k\ge k_1.
\end{equation} 
Then obviously \vglue-5mm
\begin{equation}\label{r10}
      \sup_{x\in X_k, k\ge k_1} \frac {r_x}{1-|x|} \, \le \,\frac 1{100}\, .
\end{equation} 
Moreover,  by our choice of $X_k$ and the fact that the sequence $(a_k)$ is decreasing,  
 the balls $\ov B(x,r_x)$, $x\in X_k$, $k\ge k_1$, are pairwise disjoint. So, omitting
the set of all $\ov B(x,r_x)$, $x\in X_k$, $k\ge k_1$, from the unit ball, we obtain
a champagne subregion.

\section{Result based on  ``one-bubble-estimates''}   

It may be surprising that, having Proposition \ref{nlogn},  already the ``one-bubble-approach'' 
 of \cite[Proposition 3]{HN-sigma},  which only uses the global Green function with one pole 
 and the minimum principle, immediately yields a  result  which  
is  almost as strong as Theorem C.   
Let us take        
$$
            m_k:=1, \quad   \und  \a_k:= \frac 1{k},  \und \b_k:=(\log k)^2
$$
so that $ a_k:= (k(\log k)^2)\inv$. 

\begin{proposition}\label{bubble}  Let $U$ be the unit ball,   $\ve>1/(d-1)$, and $\delta>0$. Then 
there exists $k'\ge k_1$ such that  the union~$A$ of all closed balls $\ov B(x,r_x)$, $x\in X_k$, $k\ge k'$, 
 is unavoidable  and 
\footnote{If $d\ge 3$, then $\vp(r_x)^{1+\ve}=r_x^{(d-2)(1+\ve)}$, where the critical exponent
$(d-2)(1+1/(d-1))=d-1-1/(d-1)$ is strictly smaller than $d-1$.}
\begin{equation}\label{s-weak}
          \sum_{x\in X_A}  \vp(r_x)^{1+\ve}\ < \  \delta. 
\end{equation}                               
\end{proposition}

\begin{proof} 
By (\ref{log-estimate}) and (\ref{ar}), there exists a constant $c>0$ such that, for every $k\ge k_1$, 
$$
\# X_k \, \vp(r_k)^{1+\ve} \le c \, \frac {(\log k)^2} {k^{\ve(d-1)}}\,.
$$
So (\ref{s-weak}) holds,   if $k'$ is sufficiently large.

We next claim that, for every $k\ge k_1$, the set $A$ is unavoidable. Indeed, let us fix  $k\ge k_1$ and $z\in\partial U_k$.
There exists $x\in B(z,a_k/3)\cap X_k$.  Let   $E:=\ov B(x,r_k)$. If $z\in E$, then   $ H_{U_{k+1}\setminus E} 1_E(z)=1$.
So let us assume that $z\notin E$ and let  
\begin{equation*} 
                           g(y):=\vp(r_k) \bigl(N(|y-x|)-\vp(a_k)\inv \bigr) , \qquad y\in\reald.
\end{equation*} 
Since $B(x,a_k)\subset U_{k+1}$, we know that $g\le 0$ on~$ \partial U_{k+1}$. 
Moreover,  $g\le 1$ on the boundary of~$E$.   By the minimum principle,
$$
           H_{U_{k+1}\setminus E} 1_E\ge g \on U_{k+1}\setminus E.
$$
Considering  the cases $d=2$ and $d\ge 3$ separately, we get that
$g(z) >\vp(r_k) a_k^{2-d}= \a_k$ and hence $H_{U_{k+1}\setminus E} 1_E(z)>\a_k$.

By Proposition \ref{nlogn}  and the divergence
 of  the series $\sum \a_k$, we conclude that $A$~is unavoidable. 
\end{proof} 

\section{Main result for the unit ball}

From now on we shall assume that $k_2\in\nat$, $k_2\ge k_1$, and  $f$ is a strictly positive increasing real function on   
$\bigl(0,m_{k_2}^{1-d}\,\bigr]$ such that 
\begin{equation}\label{f-def}
 \sum\nolimits_{k\ge k_2} \b_k f(\a_k^{d-1})<\infty.
\end{equation} 
We recall that, by (\ref{ar}),          $ \a_k^{d-1}\le m_k^{1-d}$ and hence $f(\a_k^{d-1})\le f(m_k^{1-d})$.

\begin{examples}{\rm
1. Given $\ve>0$, we may  choose $M\in\nat$,  
$M>1+2((d-1)\ve)\inv$, and define 
\begin{equation*} 
m_k:=k^M , \qquad \a_k:=k^{-(M+1)}, \qquad  \b_k:=k, \qquad f(t):=t^\ve,
\end{equation*} 
 since then  $\b_k f(\a_k^{d-1})=k^{-((M+1)(d-1)\ve-1)}$, where $(M+1)(d-1)\ve-1>1$.
An application of Theorem \ref{best} will then yield the results by Gardiner/Ghergu and Pres (Theorem C).

2. For a proof of Theorem \ref{main} we take 
\begin{equation*} 
   m_k:=p_n(k), \qquad  \a_k:= (k m_k)\inv, \qquad    \b_k:=k,   \qquad f(t):=\bigl(\log^{(n)} \frac 1t\bigr)^{-3}
\end{equation*} 
 (see Remark \ref{best-choice}). Then (\ref{f-def}) holds, since 
\begin{equation*} 
                                     \b_k  f(m_k^{1-d})= k\cdot (\log^{(n)} m_k^{d-1}\bigr)^{-3}\approx k^{-2}.
\end{equation*} 
}
\end{examples}

We now establish the following  general result (where $f$ denotes an arbitrary function satisfying (\ref{f-def})).
It implies Theorem \ref{main} for the unit ball, since $(\log^{(n)}(t))\inv \le  (\log^{(n+1)} (t))^{-3} $ for large~$t$.

\begin{theorem}\label{best}
Let $U$ be the unit ball and $\delta>0$. Then there exists $k'\ge k_2$ such that
the   union $A$ of all $ \ov B(x,r_x)$, $x\in X_k$, $k\ge k'$,  is  unavoidable, 
$U\setminus A$ is a~champagne subregion of $U$,  and 
\begin{equation}\label{best-est}
\sum\nolimits_{x\in X_A}  \vp(r_x) f (\vp(r_x))  < \delta.
\end{equation} 
 \end{theorem}

We may immediately note that, by (\ref{log-estimate}) and (\ref{ar}),
\begin{equation}\label{k-est}
 \# X_k \, \vp(r_k)  f(\vp(r_k))\le C \b_k  f(\a_k^{d-1}),   \qquad k\ge k_2
\end{equation} 
($C$ being some constant). So it remains to show that $A$ is unavoidable.

To that end we introduce intermediate balls:
For every $k\ge k_2$, let $M_k:= k_2+\sum_{k_2\le i< k} m_i$, and for every  $1\le j\le m_k$ and $n:=M_k+j$,   let
\begin{equation*} 
  \g_n:=\a_k,\qquad  \rho_n:=R_k+j a_k \und  V_n:= B(0,\rho_n).
\end{equation*} 
Then $\sum_{n\ge k_2} \g_n=\sum_{k\ge k_2} m_k\a_k=\infty$,  $V_{M_k}=U_k$ and $V_{M_{k+1}}=U_{k+1}$. 
Moreover,  $\ov V_n\subset V_{n+1}$,  for every $n\in\nat$, and $V_n\uparrow B(0,1)$.

Estimates  for the potentials of the measures
\begin{equation*} 
                                \mu_k:=a_k^{d-1} \sum\nolimits_{x\in X_k} \ve_x, \qquad k\ge k_2,
\end{equation*} 
(note that $\|\mu_k\|\approx 1$, by (\ref{area})) on the  balls $V_n$, $n=M_k+1,M_k+2,\dots, M_k+m_k=M_{k+1}$,
will yield the following Lemma,  which combined with Proposition \ref{nlogn} finishes the proof of Theorem \ref{best}.

\begin{lemma}\label{important} There exists $c>0$ such that the following holds:
Let  $k\ge k_2$,  let $E$ denote the union of all $\ov B(x,r_x)$, $x\in X_k$, and let 
$0\le j<m_k$,   $n:=M_k+j$.            
Then
\begin{equation}\label{est-important}
                      H_{V_{n+1}\setminus E}1_{E} \ge c \g_n  \on \ov V_n.
\end{equation}
\end{lemma} 

\begin{proof} 
Let $a:=a_k$, $\g:=\g_n=\a_k$, $r:=r_k$, and $\mu:=\mu_k$. Let $G$ denote the Green function for $V_{n+1}$
 and $\sigma$ the normalized surface measure on $\partial U_k$.
Observing that 
$$
G\sigma=G(\cdot,0)=\vp(|\cdot|)\inv-\vp(\rho_{n+1})\inv  \on V_{n+1}\setminus U_{k}, 
$$
$1/2\le \rho_k<\rho_{n+1}<1$,  $\rho_{n+1}-\rho_{k}=(j+1) a$, and $\rho_{n+1}-\rho_n= a$, 
we obtain 
\begin{equation*} 
G\sigma\approx  a\on \partial V_n \und   G\sigma\approx   (j+1)a \on  \ov B(x,r).
\end{equation*} 
For the moment we fix  $x\in X_k$. Since $X_k$ is almost uniformly distributed and $\|\mu\|\approx 1$,
we get  that, for some constant $C=C(d)\ge 1$,        
 \begin{equation*}
G\mu \ge   C\inv a \on \partial V_n \und 
G\mu\le  a^{d-1}G(\cdot,x)+ C  (j+1)a  \on \ov B(x,r) .
\end{equation*} 
Let $y\in\partial B(x,r)$. If $d\ge 3$, then $\vp(r) G(y,x)< \vp(r) |x-y|^{2-d} =1$. If $d=2$, then 
$G(y,x)\le G_{B(x,2)} (y,x)=\log (2/|x-y|)$, and hence $\vp(r) G(y,x)<2$. Therefore 
\begin{equation*} 
                           \g   G\mu(x)\le 2\g a^{d-1}\vp(r)\inv + C\g (j+1)a=(2+C(j+1)\g)a\le (2+C)a
\end{equation*} 
(where the last inequality follows from $m_k\g\le 1$).

So we know that $\g G\mu\le (2+C)a$ on  $\partial E$  and $G\mu\ge C\inv a$ on $\partial V_n$. 
By the minimum principle, (\ref{est-important}) follows with $c:=C\inv (2+C)\inv $. 
\end{proof} 

For an application in Section \ref{final} let us note the following.

\begin{corollary}\label{y-ball}
Let $y\in \reald$, $0<r< R\le 1$, $\g\in (0,1)$,  and \hbox{$\delta_y>0$}. Then  there exist a finite set $X_y$
 in  $B(y,R)\setminus \ov B(y,r)$ and  $0<s_x\le (R-|x-y|)/100$, $x\in X_y$, such that the closed balls $\ov B(x,s_x)$, $x\in X_y$,
are pairwise disjoint,
\begin{equation}\label{y-ball-fill} 
                 \sum\nolimits_{x\in X_y} \vp(s_x) f(\vp(s_x)) <\delta_y
\und
                  H_{B(y,R)\setminus A_y} 1_{A_y} \ge \g \on \ov B(y,r),
\end{equation} 
where $A_y$ is the union of all $\ov B(x,s_x)$, $x\in X_y$.
\end{corollary} 

\begin{proof} By translation and scaling invariance, it suffices to consider the case, where $B(y,R)$ is the unit ball.   
 By (\ref{k-est}), there exist $ k', k''\in \nat$, $k_2\le k'<k''$,  such that $R_{k'}> r$,   
$$
               \sum_{k\ge k'}  \# X_{j} \vp(r_k) f(\vp(r_k)) <\delta_0, \und
       \prod_{M_{k'}\le n< M_{k''}} (1-c\g_n) \le 1-\g
$$
($c$ being the constant from Lemma \ref{important}).  Let $X_0$ denote the union of all $X_k$, $k'\le k<k''$. 
Then, by Lemma \ref{important} and Proposition~\ref{nlogn},   
the proof is finished taking $s_x:=r_x$, $x\in X_0$
(note that  (\ref{HUE}) trivially holds for $n\ge M_{k''}$ with $\g_n:=0$ and $E:=\emptyset$).    
\end{proof} 

\section{Main result for arbitrary connected open sets}\label{final}
 
Clearly, Theorem \ref{main} for arbitrary connected open sets   is a consequence of the following  general result,
where $f$ is a function as considered in Section 5.

\begin{theorem}\label{general-best}
Let $U\ne \emptyset$ be a connected open set in $\reald$, $d\ge 2$,  and \hbox{$\delta>0$}.\footnote{Even 
$U=\reald$ is admitted, but of interest only if $d\ge 3$ (for $d=2$ any small closed disc is unavoidable).}
Then there exists  a~champagne subregion $U\setminus A$ of $U$ such that $A$ is unavoidable and
\begin{equation*} 
\sum\nolimits_{x\in X_A}  \vp(r_x) f (\vp(r_x))  < \delta.
\end{equation*} 
 \end{theorem} 

\begin{proof} 
Let us choose  bounded open sets $V_n\ne\emptyset$, $n\in\nat$, such that $\ov V_n\subset V_{n+1} $
and $V_n\uparrow U$. For every $n\in\nat$, we define 
$$
              b_n:=\min\bigl\{\frac 1n, \frac 12 \dist(\partial V_n, \partial V_{n-1}\cup \partial V_{n+1})\bigr\}
$$
(take $V_0:=\emptyset$) and choose a finite subset $Y_n$ of $\partial V_n$ such that the balls $B(y,b_n/2)$, 
$y\in Y_n$, cover $\partial V_n$ and the balls $B(y,b_n/6)$, $y\in Y_n$, are pairwise disjoint. 
For $y\in Y_n$, let \vglue-5mm
\begin{equation}\label{deltan}
 \delta_y:=\frac \delta{\# Y_n\cdot 2^n},
\end{equation} 
let $X_y$ be a finite set in $B(y,b_n/6)\setminus \ov B(y,b_n/7)$ and $0<s_x\le (b_n/6-|x-y|)/100$, $x\in X_y$,
such that (\ref{y-ball-fill}) holds with $\g=1/2$ and $r=b_n/7$ (see Corollary \ref{y-ball}). 

Let $X$ be the union of all $X_y$, $y\in Y_n$, $n\in\nat$. Of course,  for all $x\in X$,
$s_x/\dist(x,U^c)<1/100$, and $s_x\to 0$ if $x\to\infty$. Moreover,  by (\ref{y-ball-fill}) and (\ref{deltan}),
$$
    \sum\nolimits_{x\in X} \vp(s_x) f(\vp(s_x))  <\sum\nolimits_{n\in\nat} \sum\nolimits_{y\in Y_n}\delta_y=\delta.
$$
So it remains only to prove that the union $A$ of all $\ov B(x,s_x)$, $x\in X$, 
is unavoidable.  

To that end we define \footnote{Of course $2 \eta$ is easily determined: It is $\log 2/\log 7$, if $d=2$, and
$(2^{d-2}-1)/(7^{d-2}-1 )$, if $d\ge 3$.}  
\begin{equation*} 
      \eta:=\frac 12\,  \inf \{H_{B(0,1)\setminus \ov B(0,1/7)} 1_{\ov B(0,1/7)} (z)\colon |z|\le 1/2\}.
\end{equation*} 
Let us fix $n\in\nat$ and $z\in \partial V_n$.  There exists $y\in Y_n$ such that $|z-y|<b_n/2$. Let
$E$ be the union of all $\ov B(x,s_x)$, $x\in X_y$.  
We claim that
\begin{equation}\label{eta-est}
             H_{V_{n+1}\setminus E} 1_E (z)\ge \eta.
\end{equation} 
Then Proposition \ref{nlogn} (this time with $\a_n:=\eta$) will show that $A$ is unavoidable.

To prove the claim let $B:=B(y,b_n)$, $B':=B(y, b_n/6)$, and $F:=\ov B(y,b_n/7)$. By the minimum principle,   
$$
H_{V_{n+1}\setminus E}1_E\ge H_{B\setminus E} 1_E\ge  H_{B'\setminus E} 1_E,
$$
where $H_{B'\setminus E}1_E\ge 1/2$ on $F$, and hence $H_{B\setminus E}1_E\ge (1/2) H_{B\setminus F} 1_F$.
By translation and scaling invariance, we thus conclude that
$$
H_{V_{n+1}\setminus E}1_E(z) \ge \frac 12  H_{B\setminus F} 1_F(z) \ge \eta,
$$
that is, (\ref{eta-est}) holds and our proof is finished.
\end{proof} 

\bibliographystyle{plain} 

\def\cprime{$'$} \def\cprime{$'$}

{\small \noindent 
Wolfhard Hansen,
Fakult\"at f\"ur Mathematik,
Universit\"at Bielefeld,
33501 Bielefeld, Germany, e-mail:
 hansen$@$math.uni-bielefeld.de}\\
{\small \noindent Ivan Netuka,
Charles University,
Faculty of Mathematics and Physics,
Mathematical Institute,
 Sokolovsk\'a 83,
 186 75 Praha 8, Czech Republic, email:
netuka@karlin.mff.cuni.cz}

\end{document}